\documentclass[11pt]{amsart}

\usepackage{epsfig}

\usepackage{graphicx}
\usepackage{amscd}
\usepackage{amsmath}
\usepackage{amsfonts}
\usepackage{amssymb}
\usepackage{verbatim}

\newtheorem{theorem}{Theorem}[section]
\theoremstyle{plain}

\newtheorem{corollary}[theorem]{Corollary}

\newtheorem{defi}[theorem]{Definition}

\newtheorem{lemma}[theorem]{Lemma}
\newtheorem{notation}[theorem]{Notation}

\newtheorem{prop}[theorem]{Proposition}

\numberwithin{equation}{section}

\def\nuhat{\widehat{\nu}}
\def\muhat{\widehat{\mu}}

\def\Hk{{\mathcal H}}

\def\wtil{\widetilde}
\def\half{\frac{1}{2}}

\newcommand{\lam}{\lambda}

\newcommand{\gam}{\gamma}

\newcommand{\Gam}{\Gk}

\def\b0{{\bf 0}}

\newcommand{\R}{{\mathbb R}}

\newcommand{\Z}{{\mathbb Z}}

\def\N{{\mathbb N}}

\newcommand{\E}{{\mathbb E}\,}
\newcommand{\Prob}{{\mathbb P}\,}
\def\P{\Prob}

\def\bp{{\bf p}}
\def\wh{\widehat}

\def\Dk{{\mathcal D}}

\def\be{\begin{equation}}
\def\ee{\end{equation}}
\newcommand{\Ek}{{\mathcal E}}
\newcommand{\Fk}{{\mathcal F}}
\def\Gk{{\mathcal G}}

\newcommand{\eps}{{\varepsilon}}

\newcommand{\const}{{\rm const}}

\def\ve1{\vec{1}}

\def\Ak{{\mathcal A}}

\def\lb{\boldsymbol{\lam}}
\def\gb{\boldsymbol{\gam}}
\def\pb{\boldsymbol{p}}
\def\ab{\boldsymbol{a}}

\def\nb{\boldsymbol{n}}
\def\eb{\boldsymbol{e}}
\def\vb{\boldsymbol{v}}
\def\Bb{\boldsymbol{b}}
\def\qb{\boldsymbol{q}}

\begin{document}

\title[Fourier decay for self-similar measures]
{Fourier decay for self-similar measures}

\author[B.\ Solomyak]{BORIS SOLOMYAK}

\address{Boris Solomyak, Department of Mathematics, Bar-Ilan University, Ramat Gan, Israel}
\email{bsolom3@gmail.com}

\thanks{Supported in part by the Israel Science Foundation grant 396/15.}

\date{\today}

\begin{abstract}
We prove that, after removing a zero Hausdorff dimension exceptional set of parameters, all self-similar measures on the line have a power decay of the Fourier transform at infinity. 
In the homogeneous case, when all contraction ratios are equal, this is essentially due to Erd\H{o}s and Kahane. In the non-homogeneous case the difficulty we have to overcome is the apparent lack of convolution structure.
\end{abstract}

\maketitle

\section{Introduction}

\thispagestyle{empty}

For a finite positive Borel measure $\mu$ on $\R$, consider the Fourier transform
$$
\wh\mu(t) = \int_\R e^{itx}\,d\mu(x).
$$
The behavior of the Fourier transform at infinity is an important issue in many areas of mathematics. The measure $\mu$ is called a {\em Rajchman measure} if 
$\lim_{|t|\to \infty} \wh\mu(t)=0$. 
Riemann-Lebesgue Lemma says that absolutely continuous measures are Rajchman, but which singular measures are Rajchman is a subtle question with a long history, see \cite{Lyons}.
For many purposes simple convergence of $\muhat(t)$ to zero is not enough, and some quantitative decay is needed. 

\begin{defi}{\em
For $\alpha>0$ let
$$
\Dk(\alpha) = \bigl\{\nu\ \mbox{finite positive measure on $\R$}:\ \left|\widehat{\nu}(t)\right| = O(|t|^{-\alpha}),\ \ |t|\to \infty\bigr\},
$$
and denote $\Dk = \bigcup_{\alpha>0} \Dk(\alpha)$. A measure $\nu$ is said to have {\em power Fourier decay} if $\nu\in \Dk$.}
\end{defi}

This property has a number of applications: for instance, if $\mu$ has power Fourier decay, then $\mu$-almost every number is normal to any base, see \cite{DEL,QR}, and the support of $\mu$ has positive Fourier dimension, see \cite{Mattila}.

\medskip

In this paper we focus on the most basic class of ``fractal measures,'' namely, self-similar measures on the line. 

\begin{defi}{\em
Let $m\ge 2$, $\lb=(\lam_1,\ldots,\lam_m)\in (0,1)^m$, $\Bb=(b_1,\ldots,b_m)\in \R^m$, and let $\qb=(q_1,\ldots,q_m)$ be a probability vector. The Borel probability measure $\mu = \mu_{\lb,\Bb}^{\qb}$ on $\R$ satisfying 
\be \label{eq-ssm}
\mu = \sum_{j=1}^m q_j (\mu\circ f_j^{-1}),\ \ \ \mbox{where}\ f_j(x) = \lam_j x + b_j,
\ee
is called {\em self-similar}, or {\em invariant}, for the iterated function system (IFS) $\{f_j\}_{j=1}^m$, with the probability vector $\qb$. It is well-known  that there exists a unique such measure 
\cite{Hutch}.
We assume that the fixed points ${\rm Fix}(f_j) = \frac{b_j}{1-\lam_j}$ are not all equal (otherwise, the measure $\mu$ is a point mass) and call the corresponding pairs $(\lb,\Bb)$ {\em non-trivial.} We write $\qb>0$ if $q_j>0$ for all $j$.}
\end{defi}

\begin{theorem} \label{th-main}
For $m\ge 2$, there exists a set $\Ek$ of zero Hausdorff dimension in $(0,1)^m$ such that for all $\lb \in (0,1)^m\setminus \Ek$, for all non-trivial $(\lb,\Bb)$ and for all  $\qb>0$ we have $ \mu_{\lb,\Bb}^{\qb}\in \Dk$.
\end{theorem}

The theorem is an immediate consequence of the following:

\begin{theorem} \label{th-main1}
Fix $m\ge 2$, $1 < C_1 < C_2<\infty$ and $\epsilon,s>0$. Then there exist $\alpha>0$ and $\wtil{\Ek}\subset (0,1)^m$, depending on these parameters, such that $\dim_H(\wtil{\Ek})\le s$ and for all $\lb\in (0,1)^m\setminus
\wtil{\Ek}$, with
\be \label{cond0}
0 <C_2^{-1} \le \min_j \lam_j \le \max_j \lam_j \le C_1^{-1} < 1,
\ee
$(\lb,\Bb)$ non-trivial, and $\qb$ satisfying $\min_j q_j \ge \epsilon$, we have ${\mu}_{\lb,\Bb}^{\qb}\in \Dk(\alpha)$.
\end{theorem}

We do not attempt to give specific quantitative estimates of the decay rate, although in principle, this is possible. Our proof gives extremely slow power decay.

\medskip

Theorem~\ref{th-main} should be compared with a recent result of Li and Sahlsten, which was an inspiration for us.

\begin{theorem}[\cite{LS1}] \label{th-LS}
Let ${\mu}_{\lb,\Bb}^{\qb}$ be a self-similar measure with non-trivial $(\lb,\Bb)$ and  $\qb>0$.

{\rm (i)} If there exist $i\ne j$ such that $\log\lam_i/\log\lam_j$ is irrational, then $\widehat{\mu}_{\lb,\Bb}^{\qb}(t)\to 0$ as $|t|\to \infty$.

{\rm (ii)} If $\log\lam_i/\log\lam_j$ is Diophantine for some $i\ne j$, then there exists $\alpha>0$ such that
$$
\left|\widehat{\mu}_{\lb,\Bb}^{\qb}(t)\right| = O\left(|\log|t||^{-\alpha}\right),\ \ |t|\to \infty.
$$
\end{theorem}

The methods are quite different: \cite{LS1} uses an approach based on renewal theory, whereas we develop a multi-parameter generalization of the so-called ``Erd\H{o}s-Kahane argument''.

\subsection{Background}
The best-known case is {\em homogeneous}, when all the contraction ratios are equal: $\lam_j = \lam,\ j\le m$. 
There is a vast literature devoted to it, so we will be brief. An important class of examples is the family of Bernoulli convolutions $\nu_\lam$, which is defined as the invariant measure for the IFS $\{\lam x -1, \lam x+1\}$, with $\lam\in (0,1)$ and probabilities $\{\half,\half\}$. 
One of the original motivations for studying $\nuhat_\lam$ was the problem: {\em for which $\lam\in (\half,1)$ is $\nu_\lam$ singular/absolutely continuous?} (it follows from the ``Law of Pure Type" that $\nu_\lam$ cannot be of mixed type \cite{JW}). Erd\H{o}s \cite{Erd1} proved that $\nuhat_\lam(t)\not\to 0$ as $t\to \infty$ when $\theta=1/\lam$ is a {\em Pisot number}, 
hence the corresponding $\nu_\lam$ is singular.
Recall that a Pisot number is an algebraic integer greater than one whose algebraic (Galois) conjugates are all less than one in modulus. Later Salem \cite{Salem} showed that if
$1/\lam$ is not a Pisot number, then  $\nuhat_\lam(t)\to 0$ as $t\to \infty$, thus providing a characterization of Rajchman Bernoulli convolution measures. In spite of the recent breakthrough results, 
see \cite{Hochman14,Shmerkin14,Shmerkin19,Varju1,Varju2}, the original problem of absolute continuity/singularity for $\mu_\lam$ is still open.

The first non-trivial result on absolute continuity of $\nu_\lam$ was obtained by Erd\H{o}s \cite{Erd2} in 1940. In fact, he proved that for any $[a,b]\subset (0,1)$ there exists $\alpha>0$ such that
$\nu_\lam\in \Dk(\alpha)$ for a.e.\ $\lam\in [a,b]$. Using this and the convolution structure of $\nu_\lam$, he deduced that $\nu_\lam$ is absolutely continuous for a.e.\ $\lam$ sufficiently close to 1. Later, Kahane \cite{kahane} realized that Erd\H{o}s' argument actually gives that $\nu_\lam\in \Dk$ for all $\lam\in (0,1)$ outside a set of zero Hausdorff dimension. (We should mention that  only very few specific $\lam$ are known, for which $\nu_\lam$ has power Fourier decay, found by 
Dai, Feng, and Wang \cite{DFW}.) 
The  Erd\H{o}s-Kahane result plays an important role in the proof of absolute continuity for all $\lam\in (\half,1)$ outside of a zero Hausdorff dimension set by Shmerkin
\cite{Shmerkin14,Shmerkin19}.
 The general homogeneous case is treated analogously to Bernoulli convolutions: the self-similar measure is still an infinite convolution and most of the arguments go through with minor modifications, see \cite{DFW,ShmerkinSolomyak16a}.
 An exposition of the ``Erd\H{o}s-Kahane argument'' with quantitative estimates was given in \cite{PSS00}, and then extended and generalized. Its variants were used in a number of recent papers in fractal geometry and dynamical systems, among them \cite{ShmerkinSolomyak16a,ShmerkinSolomyak16b,FalconerJin14,SSS,KaOrp,BuSo14,BuSo18}.

In the {\em non-homogeneous} not all contraction ratios are the same and the self-similar measure is not a convolution, which makes its study more difficult. First results on absolute continuity were obtained by Neunh\"auserer \cite{Neun_2001} and  Ngai and Wang \cite{NgaiWang05}. In \cite{SSS}, joint with Saglietti and Shmerkin, we proved that, given a  probability vector $\qb>0$ and vector of translations $\Bb\in \R^m$, with all components distinct, for a.e.\ $\lb\in (0,1)^m$ in the ``natural'' parameter region (which depends on $\qb$), the self-similar measure $\mu_{\lb,\Bb}^{\qb}$ is absolutely continuous. 
The proof  was based on a decomposition of the self-similar measure into an integral of measures having a convolution structure, that are only statistically self-similar.
A variant of the Erd\H{o}s-Kahane argument was used  to establish power Fourier decay for the latter (for all but a zero-dimensional set of parameters), but this was not sufficient to deduce any Fourier decay for the original self-similar measure. The methods of \cite{SSS} were pushed further by K\"aenmaki and Orponen \cite{KaOrp}, but again, no conclusion was made for the Fourier decay of non-homogeneous self-similar measures.

We note (thanks to Pablo Shmerkin for bringing this to my attention) that a measure may have power decay outside of a sparse set of frequencies, even if it is not a Rajchman measure. In fact, Kaufman \cite{Kaufman} (in the homogeneous case) and Tsujii \cite{Tsujii} (in the non-homogeneous case) proved that for any non-trivial self-similar measure $\mu$ on the real line, for any $\eps>0$ there exists $\delta>0$ such that the set
$$
\{t\in [-T,T]:\ |\widehat{\mu}(t)| \ge T^{-\delta}\}
$$
can be covered by $T^\eps$ intervals of length $1$. Mosquera and Shmerkin \cite{MS} made the dependence of $\delta$ on $\eps$ quantitative in the homogeneous case. The papers \cite{Kaufman,MS} use a version of the Erd\H{o}s-Kahane argument, whereas the proof in \cite{Tsujii} is based on large deviation estimates.

The study of Fourier decay for other classes of dynamically defined measures has been quite active recently. 
We only mention a few papers, without an attempt to be comprehensive.
Jordan and Sahlsten \cite{JS} obtained power Fourier decay for Gibbs measures for the Gauss map, using methods from dynamics and number theory. Bourgain and Dyatlov \cite{BD} established Fourier decay for Patterson-Sullivan measures associated to a convex co-compact Fuchsian group, using methods from additive combinatorics; see also \cite{SaSte,LNP}. Li \cite{Li2} proved that the stationary measure for a random walk on $SL_2(\R)$ has power decay, when the support of the driving measure generates a Zariski dense subgroup, following his earlier work \cite{Li} showing that such a measure is Rajchman. He initiated the approach based on renewal theory, which was later used by Li and Sahlsten \cite{LS1} to prove Theorem~\ref{th-LS}. Recently the same authors extended their result to a class of self-affine measures in $\R^d$ in \cite{LS2}.

The rest of the paper is devoted to the proof of Theorem~\ref{th-main1}. As already mentioned, it is based on a generalization of the Erd\H{o}s-Kahane argument, but there are many new features, mainly because we have to deal with multi-parameter families.


\section{Reduction}

In view of $(\lb,\Bb)$ being non-trivial, by a linear change of variable we can fix two translation parameters, for instance, $b_1=0$ and $b_2>0$ arbitrary (it can even depend on the other parameters; this would only change the scale on the $t$-axis, but would not affect the  rate of decay of the Fourier transform). After that, we will pass to a higher iterate $\ell$ of the IFS, which preserves the invariant measure.  The reason for doing this is to obtain an IFS with many maps having the same contraction ratio (in fact, the number of maps $m^\ell$, grows exponentially with $\ell$, whereas the number of distinct contraction ratios grows polynomially). 
In this sense, the proof resembles the strategy of the proof in \cite{SSS}, although in other aspects it is very different. 
We now formulate the main technical result.

\begin{theorem} \label{th-tech} Let $d\ge 2$, and consider the IFS
\be \label{ifs-tech}
\Fk_{\gb,\ab}^{\pb} = \{\gam_j x + a_k^{(j)}:\ j=1,\ldots,d;\ 1 \le k \le k_j\},\ \ \pb = (p_k^{(j)}:\ j\le d;\ 1\le k \le k_j),
\ee
where $\gb = (\gam_1,\ldots,\gam_d)$ is the set of {\bf distinct} contraction ratios, $k_1\ge 2$ (so the number of maps in the IFS is strictly greater than $d$), $\ab$ is a vector of translations, and $\pb$ is a probability vector. Let $\mu_{\gb,\ab}^{\pb}$ be the corresponding self-similar measure. 
Fix $\epsilon>0$ and $1 < B_1 < B_2 < \infty$. Assume that
\be \label{cond1}
0 < B_2^{-1} \le \gam_{\min} < \gam_{\max} \le B_1^{-1}<1.
\ee
Let $s>0$ be such that 
\be \label{cond2}
B_1^s > d.
\ee
Then there exist $\alpha>0$ and $\Ek'\subset (0,1)^d$, depending on $d, B_1, B_2, s,\eps$, such that $\dim_H(\Ek')\le s$ and for all $\gb\in (0,1)^d\setminus \Ek'$, satisfying (\ref{cond1}), for all $\ab$ such that
\be \label{cond3}
a_2^{(1)} - a_1^{(1)}=\pi,
\ee
 and all $\pb$ such that $\min_j p_j \ge \epsilon$, we have ${\mu}_{\gb,\ab}^{\pb}\in \Dk(\alpha)$.
\end{theorem}


\begin{proof}[Derivation of Theorem~\ref{th-main1} from Theorem~\ref{th-tech}] As already mentioned, we may assume that the original IFS $\{\lam_j x + b_j\}_{j=1}^m$ has $f_1(x) = \lam_1 x,\ f_2 (x) = \lam_2 x + b$, with $b>0$ arbitrary (we do not exclude the case $\lam_1=\lam_2$). Passing to the $\ell$-th iterate, we obtain an IFS with the number of maps equal to $m^\ell$ and the number of distinct contractions less than or equal to 
$$
d={{\ell+m-1}\choose{m-1}},
$$
which is the number of ways to write $\ell$ as a sum of $m$ non-negative integers. Among the maps of the new IFS there are
$$
f_1^{\ell-1}f_2(x) = \lam_1^{\ell-1}(\lam_2 x + b)\ \ \mbox{and}\ \ f_2 f_1^{\ell-1}(x) = \lam_2 \lam_1^{\ell-1} x + b.
$$
This way, we can let $\gam_1 = \lam_2 \lam^{\ell-1}_1$, $a_1^{(1)}= \lam_1^{\ell-1}b,\ a_2^{(1)} = b$, so that $a_2^{(1)} - a_1^{(1)}= b(1-\lam_1^{\ell-1})$, and  choose $b = \pi (1-\lam_1^{\ell-1})^{-1}$ to satisfy (\ref{cond3}).
Denote the new IFS by $\Fk_{\gb,\ab}^{\pb}$. Since the invariant measure remains unchanged when we pass to a higher iterate of the IFS, we have $\mu_{\lb,\Bb}^{\qb} = \mu_{\gb,\ab}^{\pb}$. The bounds for inverses of the contraction ratios of  $\Fk_{\gb,\ab}^{\pb}$ are
$$
B_1 = C_1^{\ell},\ \ B_2 = C_2^{\ell},
$$
and the probabilities satisfy $\min_{j,k} p_k^{(j)} \ge (\min_i q_i)^\ell \ge \epsilon^\ell$.
In order to satisfy (\ref{cond2}), we need
$$
C_1^{\ell s} > d \ \Longleftrightarrow\ \ell > \frac{\log d}{s\log C_1}\,.
$$
Since $d < (\ell+m)^m$, it is enough to choose $\ell\ge 2$ so that
$$
\ell > \frac{m\log(\ell+m)}{s\log C_1},
$$
which is certainly possible. Now we apply Theorem~\ref{th-tech} and obtain an exceptional set $\Ek'\subset (0,1)^d$ of Hausdorff dimension $\le s$, such that for all $\gb\in {[C_2^{-\ell},C_1^{-\ell}]}^d\setminus \Ek'$, for all vectors of translations $\ab$ satisfying (\ref{cond3}) and all probability vectors $\pb$, with $\min_{j,k} p_k^{(j)} \ge \epsilon^\ell$, holds
$\mu_{\gb,\ab}^{\pb}\in \Dk(\alpha)$, for some $\alpha = \alpha(d,B_1,B_2,\epsilon^\ell,s)$. 

 It remains to observe that we can recover $\lb$ from $\gb$ via a function which does not increase Hausdorff dimension. For instance, among the contraction ratios of $\Fk_{\gb,\ab}^{\pb}$ there are $\lam_1^\ell,\ldots,\lam_m^\ell$. We can project $\gb$ to these coordinates and then take $\ell$-th root component-wise, to obtain $\lb$. This map is Lipschitz outside of the neighborhood of zero of radius $C_2^{-\ell}$. We obtain an exceptional set $\wtil{\Ek}\subset (0,1)^m$ of $\lb$ as an image of $\Ek'$ under this map, and $\dim_H(\wtil \Ek) \le \dim_H(\Ek') \le s$. For all $\lb \in{[B_2^{-1},B_1^{-1}]}^m\setminus \wtil \Ek$, all $\Bb$ satisfying $b_1=0, b_2 = \pi(1-\lam_1^{\ell-1})$, and all probability vectors $\qb$, with $\min_i q_i \ge \epsilon$, we have $\mu_{\lb,\Bb}^{\qb} = \mu_{\gb,\ab}^{\pb}\in \Dk(\alpha)$. This completes the proof of the derivation.
\end{proof}

The rest of the paper is devoted to the proof of Theorem~\ref{th-tech}.

 
\section{Beginning of the Proof}
We consider the Fourier transform $\muhat(t) = \int_\R e^{itx}\,d\mu(x)$, where $\mu = \mu_{\gb,\ab}^{\pb}$ is the invariant measure for the IFS (\ref{ifs-tech}), that is,
$$
\mu = \sum_{j=1}^d \sum_{k=1}^{k_j} p_k^{(j)} \left( \mu \circ (f_k^{(j)})^{-1}\right),\ \ \mbox{where}\ \ f_k^{(j)} (x) = \gam_j x + a_k^{(j)}.
$$
It follows that
$$
\muhat(t) = \sum_{j=1}^d \Bigl(\sum_{k=1}^{k_j} p_k^{(j)} e^{ia_k^{(j)}t}\Bigr) \muhat(\gam_j t).
$$
We can estimate
\be \label{ineq0}
|\muhat(t)| \le \Bigl( \bigl|p_1^{(1)} + p_2^{(1)} e^{i (a_2^{(1)} - a_1^{(1)})t}\bigr| + \sum_{k=3}^{k_1} p_k^{(1)}\Bigr) \cdot |\muhat(\gam_1 t)| + \sum_{j=2}^d \Bigl(\sum_{k=1}^{k_j} p_k^{(j)}\Bigr) |\muhat(\gam_j t)|.
\ee
Denote
$$
p_j:= \sum_{k=1}^{k_j} p_k^{(j)},\ j=1,\ldots,d.
$$
Recall that $a_2^{(1)} - a_1^{(1)}=\pi$ by assumption, and use an elementary inequality
$$
|1 + e^{\pi i z}| \le 2 (1-\textstyle{\frac{\pi}{4}}z^2) \ \ \mbox{for}\ z \in[-1/2,1/2].
$$
We then obtain from (\ref{ineq0}), denoting by $\|t\|$ the distance from $t\in \R$ to the nearest integer:
\be \label{ineq1}
|\muhat(t)| \le p_1\Bigl(1-\frac{\pi \eps}{2}\|t\|^2\Bigr) |\muhat(\gam_1 t)| + \sum_{j=2}^d p_j|\muhat(\gam_j t)|,
\ee
using that $\min_k p_k^{(1)}\ge \eps$.

Next we introduce some notation. Let $\Ak = \{1,\ldots,d\}$. For a word $w\in \Ak^*$ let $\ell_j(w)$ be the number of $j$'s in $w$, and let $\ell(w) = (\ell_j(w))_{j=1}^d \in \Z_+^d$. 
For $\nb=(n_j)_{j=1}^d\in \Z^d_+$ we will write
$$
\gb^{\nb} = \prod_{j=1}^d \gam_j^{n_j},\ \ \bp^{\nb} = \prod_{j=1}^d p_j^{n_j},
$$
where  $\bp: = (p_1,\ldots,p_d)$. (Note that $\bp \ne \pb$; hopefully, this will not cause a confusion; in any case, we do not need $\pb$ any more.)
Further, let $w[1,i]$ be the prefix of $w$ of length $i$; if $i=0$, this is empty word, by convention.

Iterating (\ref{ineq1}) we obtain
\be \label{ineq2}
|\muhat(t)| \le \sum_{w\in \Ak^N} \bp^{\ell(w)} |\muhat(\gb^{\ell(w)}t)| \prod_{i:\ w_i=1}\Bigl(1 - \frac{\pi \eps}{2}\|\gb^{\ell(w[1,i-1])}t\|^2\Bigr).
\ee

\begin{notation}{\em
We will consider $\Z^d_+$ as the vertex set of a directed graph, with a directed edge going from $\nb\in \Z^d_+$ to each of $\nb'=\nb+\eb_j,\, j\le d$, where $\eb_j$ is  $j$'th unit vector. We will then 
write $\nb\to\nb'$. A vertex $\nb'$ is a descendant of $\nb$ of level $r\ge 1$ if there is a path of length $r$ from $\nb$ to $\nb'$ (the length of a path is the number of edges). 
We will identify a word $w\in \Ak^N$ with a path of length $N$ in $\Z^d_+$, formed by the sequence of vertices $\{\ell(w[1,i]):\ i=0,\ldots,N\}$ and denote this path by $\Gamma(w)$. It is clear that $\ell(w[1,i])\to \ell(w[1,i+1])$ for $i=0,\ldots,N-1$.

 We will write $\nb\leadsto \nb'$ if $\nb'$ is a descendant of $\nb$ and $\|\nb'-\nb\|_\infty \le 1$.
Equivalently, $\nb\leadsto \nb'$ iff $\nb' = \nb + \sum_{\kappa \in \Gam} \eb_\kappa$ for some, possibly empty, subset $\Gam\subset \Ak$. Thus $\nb\leadsto\nb'$ implies that either $\nb'=\nb$, or $\nb'$ is a descendant of $\nb$ of level $\le d$.
}
\end{notation}

\begin{defi}
Let $\rho\in (0,\half)$, $t>0$, and $\gb \in {[B_2^{-1}, B_1^{-1}]}^d$. Say that a vertex $\nb\in \Z^d_+$ is \underline{\em $(\gb,t,\rho)$-good} if $\gb^{\nb}t\ge 1$ and
$$
\|\gb^{\nb} t\|\ge \rho
$$
(recall that $\|\cdot\|$  denotes the distance from the nearest integer).

Further, say that a vertex $\nb\in \Z^d_+$ is \underline{\em ``on a $(\gb,t,\rho)$-good track''} if there exists $\nb'$ that is $(\gb,t,\rho)$-good and $\nb\leadsto \nb'$. 

Finally, we say that an \underline{edge} $[\nb,\nb']$ is \underline{$(\gb,t,\rho)$-{\em good}} if $\nb$ is $(\gb,t,\rho)$-good and $\nb' = \nb+\eb_1$. (Notice that the 1-st coordinate, corresponding to $w_i=1$, is ``special'' by construction, see (\ref{cond3}) and (\ref{ineq2}).)
\end{defi}

Consider  $t\in (B_1^{N-1},B_1^N]$. Then $\gb^{\ell(w)}t\le 1$ for all $w\in \Ak^N$, by the assumption $\gam_{\max} \le B_1^{-1}$. It follows from (\ref{ineq2}), roughly speaking, that in order to have a power decay for $\muhat(t)$ for $t$ at this scale, it is sufficient that for ``most'' (up to exponentially small number) words $w\in \Ak^N$ there is a fixed positive proportion of
$(\gb,t,\rho)$-good edges on the path corresponding to $w$, for some $\rho>0$.
With this in mind, we define the exceptional set of $\gb$ at scale $N$ as follows: 

\begin{defi}
Fix $k_1\in \N$ and $\rho>0$, and let $\Ek_N=\Ek_N(k_1,\rho)$ be the set of $\gb \in {[B_2^{-1}, B_1^{-1}]}^d$ such that there exists $t\in (B_1^{N-1},B_1^N]$ and a word $w\in \Ak^N$ with the properties:
\be \label{property1}
\#\bigl\{d+1 \le i\le N-d-1:\ \ell(w[1,i])\ \mbox{\em is ``on a $(\gb,t,\rho)$-good track''}\bigr\} \le \frac{N}{k_1}.
\ee
Further, we define the exceptional set by
$
\Ek':=\Ek'(k_1,\rho)= \limsup \Ek_N(k_1,\rho).
$
\end{defi}

Let
\be \label{def-rho}
\rho:= \frac{1}{4(1+B_2)(1+3B_2)}\,.
\ee

Theorem~\ref{th-tech} will follow, once we prove the  next two propositions:

\begin{prop} \label{prop-decay}
For all $k_1\in \N$ sufficiently large, there exists $\alpha>0$, depending on $d, k, B_1, B_2, \eps$, such that for all $\gb\in {[B_2^{-1}, B_1^{-1}]}^d\setminus
\Ek'(k_1,\rho)$, for all $\pb$, with $\min_j p_j \ge \eps$, and all $\ab$ satisfying (\ref{cond3}), we have $\mu={\mu}_{\gb,\ab}^{\pb}\in \Dk(\alpha)$.
\end{prop}

\begin{prop} \label{prop-excep}
For all $k_1\in \N$ sufficiently large we have $\dim_H(\Ek'(k_1,\rho))< s$.
\end{prop}


\section{Fourier decay for non-exceptional $\gb$}

Fix $\gb\not\in \Ek'=\Ek'(k_1,\rho)$, where $\rho$ is given by (\ref{def-rho}) and $k_1$ is fixed, sufficiently large. (A specific value for $k_1$ will be chosen in (\ref{choose-k}).) Then $\gb\not \in \Ek_N=\Ek_N(k_1,\rho)$ for all $N$ sufficiently large. Fix such an $N$. 
The condition $\gb\not\in \Ek_N$ means, by definition, that for every $t\in (B_1^{N-1},B_1^N]$ and for every $w\in \Ak^N$, the number of vertices ``on a $(\gb,t,\rho)$-good track''
on the path $\Gamma(w[d+1,N-d-1])$ is greater than $N/k_1$. Fix $t\in (B_1^{N-1},B_1^N]$.
Since $\gb$, $t$, and $\rho$ are now fixed, we will omit $(\gb,t,\rho)$ when talking about vertices and edges that are good or ``on a good track''.

We will consider $\Ak^N$ as a probability space, with the Bernoulli measure $\Prob=\bp^N$, and the ``random environment'' provided by the configuration of good vertices and edges. Let $p_{\min}:= \min_{j\le d} p_j$. 
Let us introduce the following random variables for $i=1,\ldots,N$:
\begin{itemize}
\item for $r\ge 1$, $X_i^{(r)}$ is the number of vertices on the path $\Gamma(w[1,i])$ having a good vertex among its $r$-level descendants;
\item $X_i=X_i^{(0)}$ is the number of good vertices on the path $\Gamma(w[1,i])$;
\item $Y_i$ is is the number of good edges on the path $\Gamma(w[1,i])$.
\end{itemize}

Notice that for every vertex of $\Gamma(w[d+1,N])$ that is ``on a good track'', there is a vertex of $\Gamma(w)$ that had a good vertex among its $d$-level descendants, and this mapping is at most $(d+1)$-to-$1$. It follows that, with probability one,
\be \label{flug1}
X_N^{(d)} \ge  \frac{N}{(d+1)k_1}\,.
\ee

\begin{lemma} \label{lem-proba}
There exist $\delta>0$, $C'>0$, and $c>0$, depending only on $p_{\min}$ and $k_1$, such that, assuming $N$ is sufficiently large (depending only on $p_{\min}$ and $k_1$), holds
$$
\P\left(Y_N < \delta N \right) \le C'\exp(-c N).
$$
\end{lemma}

We first deduce  power Fourier decay for $\gb\not\in \Ek'$ from the lemma. 

\begin{proof}[Proof of Proposition~\ref{prop-decay}]
Consider the sum in the inequality (\ref{ineq2}) and split it according to whether $Y_N=Y_N(w) <\delta N$ or $\ge \delta N$.  The sum over $w$ such that $Y_N(w)<\delta N$, is bounded by $\P\left(Y_N < \delta N \right)$. If $w$ is such that $Y_N\ge \delta N$, then the corresponding term in the right-hand side of (\ref{ineq2}) is estimated from above by $\bigl(1-\frac{\pi \eps}{2}\rho^2\bigr)^{\delta N}$, by the definition of a good edge (we also use the fact that $|\muhat(t)|\le 1$, since $\mu$ is a probability measure).  Then Lemma~\ref{lem-proba} implies, for $N$ sufficiently large:
\be \label{want7}
|\muhat(t)| \le \Bigl(1-\frac{\pi \eps}{2}\rho^2\Bigr)^{\delta N} + C'\exp(-c N).
\ee
Since $N$ was arbitrary, sufficiently large, and $t$ arbitrary in $(B_1^{N-1},B_1^N]$, this implies that $\mu\in \Dk(\alpha)$ for some $\alpha>0$.
\end{proof}

As a step in the proof of Lemma~\ref{lem-proba}, we will first establish the following

\begin{lemma} \label{lem-proba1} There exist $\delta_r>0$, $C_r'$, and $c_r>0$, for $r=0,\ldots, d$, such that, for $N$ sufficiently large,
\be \label{want8}
\P\left(X_{N-1}^{(r)} \le \delta_r N \right) \le C_r'\cdot\exp(-c_r N).
\ee
\end{lemma}

\begin{proof}[Proof of Lemma~\ref{lem-proba1}] We will show this by induction in $r$, going  from $r=d$ to $r=0$. For $r=d$ the claim trivially holds, by (\ref{flug1}).
Fix $r\in\{ 0,\ldots,d-1\}$ and assume that (\ref{want8}) holds for $r+1$. Consider the sequence of random variables
$$
Z^{(r)}_i:=\frac{X^{(r)}_{i+1}}{p_{\min}} - X_i^{(r+1)},\ \ \ i=1,\ldots,N-1.
$$
We claim that this a submartingale; in fact,
\be\label{mart1}
\E[Z^{(r)}_i\vert Z^{(r)}_{i-1}]\ge Z^{(r)}_{i-1}.
\ee
Indeed, we have either (a) $X_i^{(r+1)}=X_{i-1}^{(r+1)}$, or  (b) $X_i^{(r+1)}=X_{i-1}^{(r+1)}+1$. The former case occurs when $\ell(v[1,i])$ has no good descendants of level $r+1$, and then
$\ell(v[1,i+1])$ has no good descendants of level $r$. Thus in case (a) we have $X_{i+1}^{(r)} = X_i^{(r)}$ and $Z_i^{(r)} = Z_{i-1}^{(r)}$.

In case (b), on the other hand, $\ell(v[1,i])$ has a good descendant of level $r+1$, and then $\ell(v[1,i+1])$ has a good descendant of level $r$, with 
probability $\ge p_{\min}$, independently of the past. Then  either  $X^{(r)}_{i+1}=X^{(r)}_i$ or $X^{(r)}_{i+1}=X^{(r)}_i+1$, hence $Z^{(r)}_i = Z^{(r)}_{i-1}-1$ or $Z^{(r)}_i = Z^{(r)}_{i-1}+\frac{1}{p_{\min}}-1$. 

Formally, we obtain
\begin{eqnarray*}
\E\bigl[Z^{(r)}_i\vert Z^{(r)}_{i-1}\bigr] & = & \E\bigl[Z^{(r)}_i\,\vert \,Z^{(r)}_{i-1},\ X^{(r+1)}_i=X^{(r+1)}_{i-1}\bigr] \cdot \P[X_i^{(r+1)}=X_{i-1}^{(r+1)}] \\[1.2ex]
& + & \E\bigl[Z^{(r)}_i\,\vert \,Z^{(r)}_{i-1},\ X^{(r+1)}_i=X^{(r+1)}_{i-1}+1\bigr] \cdot \P[X_i^{(r+1)}=X_{i-1}^{(r+1)}+1]\\[1.2ex]
& = & Z^{(r)}_{i-1}\cdot \P[X_i^{(r+1)}=X_{i-1}^{(r+1)}]  + (Z^{(r)}_{i-1}-1)\cdot \P[X_i^{(r+1)}=X_{i-1}^{(r+1)}+1,\,X^{(r)}_{i+1}=X^{(r)}_i] \\[1.2ex]
& + & (Z^{(r)}_{i-1}+p_{\min}^{-1}-1)\cdot \P[X_i^{(r+1)}=X_{i-1}^{(r+1)}+1,\,X^{(r)}_{i+1}=X^{(r)}_i+1].
\end{eqnarray*}
Since $\P[X_i^{(r+1)}=X_{i-1}^{(r+1)}] + \P[X_i^{(r+1)}=X_{i-1}^{(r+1)}+1]=1$, we have
\begin{eqnarray*}
\E\bigl[Z^{(r)}_i\vert Z^{(r)}_{i-1}\bigr] -Z^{(r)}_{i-1} & = & - \P[X_i^{(r+1)}=X_{i-1}^{(r+1)}+1,\,X^{(r)}_{i+1}=X^{(r)}_i]\\[1.2ex]
& + & (p_{\min}^{-1}-1)\cdot \P[X_i^{(r+1)}=X_{i-1}^{(r+1)}+1,\,X^{(r)}_{i+1}=X^{(r)}_i+1]\\[1.2ex]
& = & -\P[X_i^{(r+1)}=X_{i-1}^{(r+1)}+1] \times \\[1.2ex] 
& \times & \bigl(-1+p_{\min}^{-1} \cdot \P[X^{(r)}_{i+1}=X^{(r)}_i+1\,|\,X_i^{(r+1)}=X_{i-1}^{(r+1)}+1]\bigr)\ge 0,
\end{eqnarray*}
confirming the claim that $\{Z^{(r)}_i\}$ is a submartingale.

We are going to apply the Azuma-Hoeffding inequality, which says that, given that $\{Z^{(r)}_i\}$ is a submartingale, if $|Z^{(r)}_i-Z^{(r)}_{i-1}|\le \alpha_i$ for all $i$, then
$$
\P\left(Z^{(r)}_{N-1}-Z^{(r)}_1 \le -y\right) \le \exp\left(\frac{-y^2}{2\sum_{i=2}^{N-1} \alpha_i^2}\right).
$$
See, e.g., \cite{Alon_Spencer} for the (two-sided) Azuma-Hoeffding inequality for martingales. The one-sided inequality for submartingales is proved similarly, see e.g., \cite{Chung_Lu}.

We have $|Z^{(r)}_i-Z^{(r)}_{i-1}|\le p_{\min}^{-1}$, hence taking $y=\frac{\delta_{r+1}N}{3}$ yields
$$
\P\left(Z^{(r)}_{N-1}-Z^{(r)}_1 \le -\frac{\delta_{r+1}N}{3}\right) \le \exp\left(\frac{-N^2p_{\min}^2\delta_{r+1}^2}{18(N-2)}\right) \le \exp\left(\frac{-Np_{\min}^2\delta_{r+1}^2}{18}\right).
$$
Since $Z_1^{(r)}$ is bounded, we have for $N$ sufficiently large:
$$
\P\left(Z^{(r)}_{N-1} \le -\frac{\delta_{r+1}N}{2}\right) \le \P\left(Z^{(r)}_{N-1}-Z^{(r)}_1 \le -\frac{\delta_{r+1}N}{3}\right) \le \exp\left(\frac{-Np_{\min}^2\delta_{r+1}^2}{18}\right).
$$
Recall that $Z^{(r)}_{N-1}= \frac{X^{(r)}_N}{p_{\min}}- X^{(r+1)}_{N-1}$, and 
$$
\P\left(X^{(r+1)}_{N-1} < \delta_{r+1}N\right)  \le C_{r+1}'\cdot \exp(-c_{r+1}N),
$$
by the inductive assumption.
Therefore for $N$ sufficiently large,
\begin{eqnarray} \nonumber 
\P\left(X^{(r)}_{N-1} \le \frac{\delta_{r+1}N\cdot p_{\min}}{3}\right) & \le &
\P\left(X^{(r)}_N \le \frac{\delta_{r+1}N\cdot p_{\min}}{2}\right) \\ & \le & \exp\left(\frac{-Np_{\min}^2\delta_{r+1}^2}{18}\right) + C_{r+1}'\cdot \exp(-c_{r+1}N), \label{ineq3}
\end{eqnarray}
and (\ref{want8}) follows.
\end{proof}

\begin{proof}[Proof of Lemma~\ref{lem-proba}]
Consider the sequence of random variables
$$
U_i := \frac{Y_{i+1}}{p_1} - X_i,\ \ i=1,\ldots,N-1.
$$
We claim that $\{U_i\}$ is a martingale; in fact,
\be \label{martin}
\E\big[U_i\vert U_{i-1}\bigr] = U_{i-1}.
\ee
 This is proved analogously to the proof of the submartingale property for $\{Z^{(r)}_i\}$ above.
 If $\nb=\ell(w[1,i])$ is not a good vertex, then the edge $[\nb,\nb']$, with $\nb' = \ell(w[1,i+1])$, is not good either, and we have $X_{i} = X_{i-1},\ Y_{i+1}=Y_i,\ U_i = U_{i-1}$.
 If, other other hand, $\nb=\ell(w[1,i])$ is a good vertex, then the edge $[\nb,\nb']$, with $\nb' = \ell(w[1,i+1])$, is good with probability $p_1$, and this is independent from the past.
Thus, if $X_i= X_{i-1}+1$, then 
$$
U_i = \left\{\begin{array}{cc} U_{i-1}-1, & \mbox{with probability}\ 1-p_1,\\
U_{i-1}+ \frac{1}{p_1} - 1, & \mbox{with probability}\ p_1. \end{array} \right.
$$
This implies (\ref{martin}); the formal computation, similar to the above, is left to the reader.

Applying the Azuma-Hoeffding inequality to $\{U_i\}$, in view of $|U_i-U_{i-1}|\le p_1^{-1}$, after a computation similar to that above, we can estimate, for $N$ large enough, using
(\ref{want8}) for $r=0$:
$$
\P\left(Y_N \le \frac{\delta_0 N p_1}{3}\right) \le \exp\left(\frac{-\delta_0^2 Np_1^2}{18}\right) + C_0'\exp(-c_0 N).
$$
This implies the desired estimate (\ref{want7}).
\end{proof}


\section{Dimension of the exceptional set}

Fix $\gb\in \Ek'$. This means that $\gb \in \Ek_N$ for infinitely many $N$. Fix such an $N$, sufficiently large. We will show that this imposes constraints on $\gb$ allowing us to construct a good cover of $\Ek_N$. By definition of $\Ek_N$, there exists $t\in (B_1^{N-1},B_1^N]$ and a word $w\in \Ak^N$, such that the number of vertices ``on a good $(\gb,t,\rho)$-track'' on the path $\Gamma(w[d+1,N-d-1])$ does not exceed $N/k_1$. Fix such a $t$ and $w\in \Ak^N$, and for $\nb\in \Z^d_+$ let
$$
\gb^{\nb} t = K_{\nb} + \eps_{\nb},\ \ K_{\nb} \in \N,\ \eps_{\nb} \in [-1/2,1/2),
$$
that is, $K_{\nb}$ is the nearest integer to $\gb^{\nb}t$ and $\|\gb^{\nb} t\| = |\eps_{\nb}|$. One should keep in mind that $K_{\nb}$ and $\eps_{\nb}$ depend on $\gb$ and $t$, but we suppress this in notation to reduce ``clutter''.

The next lemma is analogous to the ones appearing in other variants of the Erd\H{o}s-Kahane argument; see e.g.,\ \cite[Lemma 6.3]{PSS00}.

\begin{lemma} \label{lem-choice} Let $\rho$ be given by (\ref{def-rho}) and 
\be \label{def-A}
A:=  2(1+B_2) (1+3B_2)+1.
\ee
Let $\nb\in \Z^d_+$, $\nb' = \nb +\eb_j,\ \nb'' = \nb + 2\eb_j$ for some $j\in \Ak$ (in particular, $\nb\to \nb'\to \nb''$), such that $K_{\nb''}\ge 1$. The following hold:

{\em (i)} Given $K_{\nb''}$ and $K_{\nb'}$, there are at most $A$ possibilities for $K_{\nb}$.

{\em (ii)} Given $K_{\nb''}$ and $K_{\nb'}$,  the number $K_{\nb}$ is uniquely determined, provided
\be \label{cond-small}
\max\{|\eps_{\nb}|, |\eps_{\nb'}|, |\eps_{\nb''}|\} < \rho,
\ee
that is, provided none of the $\nb,\nb',\nb''$ is $(\gb,t,\rho)$-good.
\end{lemma}

\begin{proof}
We have, by assumption,
$$
\gb^{\nb} t = K_{\nb} + \eps_{\nb},\ \ \gb^{\nb'} t = \gb^{\nb}\gam_j t = K_{\nb'} + \eps_{\nb'},\ \ 
  \gb^{\nb''} t = \gb^{\nb}\gam^2_j t = K_{\nb''} + \eps_{\nb''}\ge 1.
$$
The idea is that
$$
\frac{K_{\nb}}{K_{\nb'}} \approx \gam_j^{-1} \approx \frac{K_{\nb'}}{K_{\nb''}},
$$
hence $K_{\nb}$ must be not too far from $\frac{K_{\nb'}^2}{K_{\nb''}}$. First note that
\be \label{esqu1}
\frac{K_{\nb'}}{K_{\nb''}} = \frac{\gb^{\nb'}t - \eps_{\nb'}}{\gb^{\nb'}\gam_j t - \eps_{\nb''}} \le \frac{\gb^{\nb'}t +\half}{\gb^{\nb'}\gam_j t -\half} \le \frac{3\gb^{\nb'}t}{\gb^{\nb'}\gam_jt} = 3\gam_j^{-1} \le 3B_2,
\ee
where we used the bound $1 \le \gam^{\nb''} = \gam^{\nb'}\gam_j t\le \gam^{\nb'}t$. Next,
\begin{eqnarray*}
\left|\frac{K_{\nb}}{K_{\nb'}} -\frac{K_{\nb'}}{K_{\nb''}}\right| & \le & \left|\frac{K_{\nb}}{K_{\nb'}} -\gam_j^{-1}\right| + \left|\gam_j^{-1} - \frac{K_{\nb'}}{K_{\nb''}}\right| \\[1.2ex]
& = & \frac{|\eps_{\nb} - \gam_j^{-1} \eps_{\nb'}|}{K_{\nb'}} +  \frac{|\eps_{\nb'} - \gam_j^{-1} \eps_{\nb''}|}{K_{\nb''}} \\[1.2ex]
& \le & \frac{|\eps_{\nb}| + B_2|\eps_{\nb'}|}{K_{\nb'}} + \frac{|\eps_{\nb'}| + B_2|\eps_{\nb''}|}{K_{\nb''}}\,.
\end{eqnarray*}
Therefore,
\begin{eqnarray}
\nonumber \left|K_n - \frac{K_{\nb'}^2}{K_{\nb''}}\right| & \le & \bigl(|\eps_{\nb}| + B_2|\eps_{\nb'}|\bigr) +\frac{K_{\nb'}}{K_{\nb''}}\cdot \bigl(|\eps_{\nb'}| + B_2|\eps_{\nb''}|\bigr)\\[1.1ex]
& \le & \nonumber 2(1+B_2) (1+3B_2) \cdot \max\{|\eps_{\nb}|,|\eps_{\nb'}|,|\eps_{\nb''}|\},
\end{eqnarray} 
using (\ref{esqu1}) in the last step. Now both parts of the lemma follow easily. Indeed, $K_{\nb}$ is an integer. 

(i) Since $ \max\{|\eps_{\nb}|,|\eps_{\nb'}|,|\eps_{\nb''}|\}\le \half$, once $K_{\nb'}$ and $K_{\nb''}$ are given, there are at most $A$ 
possibilities for $K_{\nb}$, see (\ref{def-A}). 

(ii) The choice of $K_{\nb}$ will be unique, provided
$$
\max\{|\eps_{\nb}|,|\eps_{\nb'}|,|\eps_{\nb''}|\} < \rho= \frac{1}{4(1+B_2)(1+3B_2)}\,
$$
see (\ref{def-rho}).
\end{proof}

\begin{corollary} \label{cor-choice}
Suppose that $\nb\to \nb'$, and we are given $K_{\vb'}$ for all $\vb'$ such that $\nb'\leadsto \vb'$; assume that all of them satisfy $K_{\vb'}\ge 1$. Then

{\rm (i)} for any $\vb$, such that $\nb\leadsto \vb$, there at most $A$ possibilities for $K_{\vb}$;

{\rm (ii)} for any $\vb$, such that $\nb\leadsto \vb$, assuming that neither $\vb$, nor any of $\vb'$, with $\nb'\leadsto \vb'$, is $(\gb,t,\rho)$-good, $K_{\vb}$ is uniquely determined. 
\end{corollary}

\begin{proof}
 Fix $\vb$ such that $\nb\leadsto \vb$. Then $\vb = \nb + \sum_{\kappa\in \Gam} \eb_\kappa$, for some $\Gam\subset \Ak$.
We have $\nb\to \nb'$; suppose $\nb'= \nb +\eb_j$ for some $j\in \Ak$. If $j\in \Gam$, then $\vb = \nb' + \sum_{\kappa\in \Gam \setminus \{j\}} \eb_\kappa$, so
$\nb'\leadsto \vb$ and $K_{\vb}$ is already known by assumption.

 If $j\not\in \Gam$, then $\vb':= \nb' + \sum_{\kappa\in \Gam} \eb_\kappa$ and $\vb'':= \nb' + \sum_{\kappa\in \Gam\cup \{j\}} \eb_\kappa$ satisfy
$$
\vb' = \vb+\eb_j,\ \vb'' = \vb'+\eb_j.
$$
Moreover,
$
\nb'\leadsto \vb',\ \nb'\leadsto \vb'',
$
so $K_{\vb'}$ and $K_{\vb''}$ are already given, and we are exactly in the situation of Lemma~\ref{lem-choice}. Applying the lemma yields the desired result.
\end{proof}

\begin{proof}[Proof of Proposition~\ref{prop-excep}]
Let $q\in \N$ be maximal, such that
\be \label{bound0}
\gb^{\ell(w[1,q])} t \ge B_2^d.
\ee
Note that
\be \label{bound1}
(N-d-1) \cdot \frac{\log B_1}{\log B_2} \le q \le N-d.
\ee
Let
$$
\nb_i:= \ell(w[1,i]),\ \ i=0,\ldots,q,
$$
be the $i$-th vertex on the path corresponding to the word $w$, so that $\nb_{i}\to \nb_{i+1}$.
We have chosen $q$ in such a way that 
$$
\gb^{\vb} t \ge 1,\ \ K_{\vb}\ge 1,\ \ \mbox{for all}\ \vb \ \mbox{such that}\ \nb_q \leadsto \vb,
$$

Recall that $w\in \Ak^N$ is fixed and the number of vertices ``on a good $(\gb,t,\rho)$-track'' on the path $\Gamma(w[d+1,N-d-1])$ does not exceed $N/k_1$.
We are going to estimate from above the number of possible configurations of integers $K_{\vb}$, where $\nb_i\leadsto \vb$ for some $i=q,\ q-1,\ldots,0$.
Note that for any $\nb\in \Z^d_+$ there are $2^d$ vertices $\vb$ such that $\nb\leadsto \vb$.

We start with the ``initial configuration'' of $K_{\vb}$ for $\vb$ such that $\nb_q\leadsto \vb$.
By the choice of $q$ we have
$$
B_2^d\le K_{\nb_q} \le B_2^{d+1},
$$
so $K_{\vb}\in [1,B_2^{d+1}]$ for all $\vb$ such that $\nb_q\leadsto \vb$.
It follows that the total number of possibilities for $K_{\vb}$ for all $\vb$ such that  $\nb_q\leadsto \vb$,  is at most
$$
L_1 := B_2^{(d+1)2^d}.
$$
Now we follow the path $\nb_q,\nb_{q-1},\ldots,\nb_{\b0}$ backwards, applying Corollary~\ref{cor-choice} at each step. Fix $i\le q$. Part (i) of the corollary says that for any $\vb$, with 
$\nb_{i-1}\leadsto \vb$, there are at most $A$ choices for $K_{\vb}$, once all the $K_{\vb'}$, with $\nb_i\leadsto \vb'$ are determined. Part (ii) of the corollary says that if none of $\nb_i$,
$\nb_{i-1}$ are on a ``good $(\gb,t,\rho)$-track'', those $K_{\vb}$ are determined uniquely. By assumption, there are no more that $N/k_1$ vertices of $w[d+1,N-d-1]$ that are ``on  a good $(\gb,t,\rho)$-track'', hence this will affect at most $2N/k_1$ transitions between $\nb_{N-d-1}$ and $\nb_{d+1}$. On each transition, we determine at most $2^d$ ``new'' values of $K_{\vb}$ (this is an ``overcount,'' but we do not try to be precise here).
If we fix the subset of $\{1,\ldots,q\}$ corresponding to the $\le N/k_1$ vertices ``on a good $(\gb,t,\rho)$-track'' on the path $\Gamma(w[d+1,N-d-1])$, we will obtain at most
$$
L_1\cdot A^{2^d[2N/k_1 + 2(d+1)]}=L_2\cdot A_1^{N/k_1}
$$
total configurations, where $ L_2 = L_1\cdot A^{2(d+1)},\ A_1 = A^{2^{d+1}}$. Taking into account all the possibilities for the subset in question and also possible values of $q\le N$ yields
that the total number of  configurations of $K_{\vb}$, for all $\vb$ under consideration, is at most
\begin{eqnarray*}
L_2\cdot N \cdot \sum_{i=1}^{\lfloor N/k_1\rfloor } {q\choose i}\cdot A_1^{N/k_1}
& < &  L_2\cdot N^2 {N\choose {\lfloor N/k_1\rfloor} }\cdot  A_1^{N/k_1}\\ &< & \exp\left(O_{B_2,d}(1)\cdot \frac{\log k_1}{k_1}\cdot N \right). 
\end{eqnarray*}

Next, note that the knowledge of all  $K_{\vb}$  associated with the path $w$ gives a good approximation of $\gam_j$. In fact, we have $\nb_{0} = \b0\leadsto \eb_j$ for $j\in \Ak$, so
$K_{\b0}$ and $K_{\eb_j}$  are among the ``known'' ones. Estimating as in Lemma~\ref{lem-choice}, we have
$$
\left|\gam_j^{-1} - \frac{K_{\b0}}{K_{\eb_j}}\right| \le \frac{1+B_2}{K_{\eb_j}},
$$
and $K_{\eb_j}\ge t\gam_j - \half\ge \frac{B_1^{N-1}}{B_2}-\half$. It follows that the knowledge of all  $K_{\vb}$ associated with the path $w$ gives a cover of the exceptional $\gb$ by balls of diameter
$\sim B_1^{-N}$. Taking into account that the number of words $w\in \Ak^N$ is equal to $d^N$, we obtain that the exceptional set $\Ek_N$ at scale $N$ may be covered by
$$
\exp\left(O_{B_2,d}(1)\cdot \frac{\log k_1}{k_1}\cdot N \right)\cdot d^N
$$
balls of diameter $\sim B_1^{-N}$. Recall that $B_1^s>d$ by (\ref{cond2}).
Thus we can choose $k_1\in \N$ such that 
\be \label{choose-k}
O_{B_2,d}(1)\cdot \frac{\log k_1}{k_1} < s\log B_1 - \log d.
\ee
Then
$$
\Hk^s(\Ek') \le \const\cdot \liminf_{N\to \infty} \exp\left[(s\log B_1 - \log d)N\right]\cdot d^N \cdot B_1^{-Ns} =\const<\infty,
$$
whence $\dim_H(\Ek') \le s$, as desired.
\end{proof}

\noindent {\bf Acknowledgement.} I am grateful to Ori Gurel-Gurevich for his help with the probabilistic argument, and to Tuomas Sahlsten for helpful discussions.

\bibliographystyle{plain}
\bibliography{nonunif}


\end{document}